\documentclass[a4paper,11pt,reqno]{amsart}
\usepackage{amsmath}
\usepackage{amstext}
\usepackage{amsbsy}
\usepackage{amsopn}
\usepackage{upref}
\usepackage{amsthm}
\usepackage{amsfonts}
\usepackage{amssymb}
\usepackage{mathrsfs}
\usepackage{mathtools}
\usepackage{times}
\usepackage{cite}
\usepackage{bm}
\usepackage{graphicx}
\usepackage{multicol}
\usepackage{color}
\usepackage{pb-diagram}
\allowdisplaybreaks
 \newtheorem{theorem}{Theorem}

 \newtheorem{lemma}[theorem]{Lemma}
\theoremstyle{definition}
 
\theoremstyle{remark}
 
\begin{document}
\title[$d$-plane transform]{The symbol of the normal operator for the $d$-plane transform on the Euclidean space}
\author[H.~Chihara]{Hiroyuki Chihara}
\address{College of Education, University of the Ryukyus, Nishihara, Okinawa 903-0213, Japan}
\email{hc@trevally.net}
\thanks{Supported by the JSPS Grant-in-Aid for Scientific Research \#23K03186.}
\subjclass[2020]{Primary 58J40, Secondary 53C65}
\keywords{$d$-plane transform, Fourier integral operator, normal operator, filtered backprojection}
\begin{abstract}
We directly compute the symbol of the normal operator for the $d$-plane transform on the Euclidean space. We show that this symbol is the product of the symbol of the power of the Laplacian of order $-d/2$ and a constant given by an invariant integral over excess-dimensional spaces. This leads to an alternative derivation of the filtered backprojection formula for the $d$-plane transform.
\end{abstract}
\maketitle
%
%
The $d$-plane transform $\mathcal{R}_d$ on the $n$-dimensional Euclidean space is an elliptic Fourier integral operator, and the inversion formula $(-\Delta_{\mathbb{R}^n})\mathcal{R}_d^\ast{\mathcal{R}_d}=c\operatorname{Id}$ with some constant $c>0$, which is called the filtered backprojection formula, is well-known, where 
$$
x=(x_1,\dotsc,x_n)\in\mathbb{R}^n, 
\quad
\Delta_{\mathbb{R}^n}
=
\frac{\partial^2}{\partial x_1^2}
+
\dotsb
+
\frac{\partial^2}{\partial x_n^2}.
$$ 
In this paper we directly compute the symbol of the composition $\mathcal{R}_d^\ast{\mathcal{R}_d}$ of Fourier integral operators $\mathcal{R}_d^\ast$ and $\mathcal{R}_d$. Theoretically the symbol is supposed to be given by the integration of the product of the symbols of $\mathcal{R}_d^\ast$ and $\mathcal{R}_d$ over $e_d$-dimensional spaces, where $e_d$ is the excess of the intersection of the canonical relations of $\mathcal{R}_d^\ast$ and $\mathcal{R}_d$. See \cite[Theorem~25.2.3]{Hoermander4}.   
\par
We now introduce the $d$-plane transform and state the inversion formula following Helgason's celebrated textbook \cite{Siggi}. Let $n$ be the space dimension of the Euclidean space not smaller than two, and let $d$ be a positive integer strictly less than $n$. Denote by $G_{d,n}$ and $G(d,n)$ the Grassmannian and the affine Grassmannian respectively, that is, $G_{d,n}$ is the set of all $d$-dimensional vector subspaces of $\mathbb{R}^n$ and $G(d,n)$ is the set of all $d$-dimensional affine subspaces in $\mathbb{R}^n$. For any $\sigma \in G_{d,n}$, we have orthogonal direct sum 
$\mathbb{R}^n = \sigma\oplus\sigma^\perp$, 
where $\sigma^\perp$ is the orthogonal complement of $\sigma$ in $\mathbb{R}^n$. 
For any fixed $\sigma \in G_{d,n}$, 
we choose a coordinate system of $\mathbb{R}^n$ such as 
$x=x^\prime+x^{\prime\prime} \in \sigma\oplus\sigma^\perp$. 
The affine Grassmannian $G(d,n)$ is given by 
$$
G(d,n)
=
\{\sigma+x^{\prime\prime}: \sigma \in G_{d,n}, x^{\prime\prime} \in \sigma^\perp\}. 
$$    
Let $O(n)$ be the orthogonal group. 
If we interpret $O(d)$ as the orthogonal group acting on $\sigma$ 
and $O(n-d)$ as the orthogonal group acting on $\sigma^\perp$, 
then we can identify $G_{d,n}$ with the quotient space 
$O(n)/(O(d){\times}O(n-d))$. 
We express $\sigma+x^{\prime\prime}$ by $(\sigma,x^{\prime\prime})$ below. 
The $d$-plane transform $\mathcal{R}_df(\sigma,x^{\prime\prime})$ of 
an appropriate function $f(x)$ on $\mathbb{R}^n$ is defined by 
$$
\mathcal{R}_df(\sigma,x^{\prime\prime})
:= 
\int_\sigma f(x^\prime+x^{\prime\prime}) dx^\prime,
$$
where $dx^\prime$ is the Lebesgue measure on $\sigma$. 
We now compute to obtain the adjoint of $\mathcal{R}_d$. 
\begin{lemma}
\label{theorem:adjoint}
The adjoint $\mathcal{R}_d^\ast$ is explicitly given by  
\begin{equation}
\mathcal{R}_d^\ast\varphi(x)
=
\operatorname{vol}(G_{d,n})
\int_{O(n)/(O(d){\times}O(n-d))}
\varphi(k\cdot\sigma,x-\pi_{k{\cdot}\sigma}x)
dk, 
\label{equation:adjoint1}
\end{equation} 
for an appropriate function $\varphi$ on $G(d,n)$, where 
$\pi_{k{\cdot}\sigma} : \mathbb{R}^n \rightarrow k\cdot\sigma$ 
is the orthogonal projection, 
$dk$ is the normalized measure which is invariant under rotations, 
$\sigma \in G_{d,n}$ is arbitrary, 
$\operatorname{vol}(X)$ denotes the volume of a compact Riemannian manifold $X$, 
$\operatorname{vol}(G_{d,n})$ is given by 
\begin{align}
  \operatorname{vol}(G_{d,n})
& =
  \frac{\operatorname{vol}\bigl(SO(n)\bigr)}{2^{d(n-d)/2} \operatorname{vol}\bigl(SO(d)\bigr) \operatorname{vol}\bigl(SO(n-d)\bigr)}
\nonumber
\\
& =
\frac{\operatorname{vol}(\mathbb{S}^{n-1}) \dotsb \operatorname{vol}(\mathbb{S}^{n-d})}{\operatorname{vol}(\mathbb{S}^{d-1}) \dotsb \operatorname{vol}(\mathbb{S}^1)},
\label{equation:volume}
\end{align}
$\mathbb{S}^k$ is the unit sphere in $\mathbb{R}^{k+1}$ with the center at the origin, $\operatorname{vol}(\mathbb{S}^k)=2\pi^{(k+1)/2}/\Gamma\bigl((k+1)/2\bigr)$, 
and $\Gamma(\cdot)$ is the gamma function. 
\end{lemma}
\begin{proof}
The formula \eqref{equation:volume} for $\operatorname{vol}(G_{d,n})$ 
is given by Abe and Yokota in \cite{AbeYokota} and Shi and Zhou in \cite[Proposition~2.2]{ShiZhou}. 
Let $f(x)$ be a compactly supported continuous function on $\mathbb{R}^n$, 
and let $\varphi(\sigma,x^{\prime\prime})$ be a compactly supported continuous function on $G(d,n)$. Pick up arbitrary $\sigma \in G_{d,n}$. 
Set $\mathcal{G}_d=O(n)/(O(d){\times}O(n-d))$ for short. 
By using the fact $\mathbb{R}^n=(k\cdot\sigma)\oplus(k\cdot\sigma)^\perp$ 
for any $k \in \mathcal{G}_d$, we deduce that 
\begin{align*}
& \int_{G(d,n)}\varphi(\sigma,x^{\prime\prime})\overline{\mathcal{R}_df(\sigma,x^{\prime\prime})}
\\
& =
  \operatorname{vol}(G_{d,n})
  \int_{\mathcal{G}_d}
  \int_{(k\cdot\sigma)^\perp}
  \varphi(k\cdot\sigma,x^{\prime\prime})
  \overline{\mathcal{R}_df(k\cdot\sigma,x^{\prime\prime})}
  dx^{\prime\prime}
  dk
\\
& =
  \operatorname{vol}(G_{d,n})
  \int_{\mathcal{G}_d}
  \int_{(k\cdot\sigma)^\perp}
  \int_{k\cdot\sigma}
  \varphi(k\cdot\sigma,x^{\prime\prime})
  \overline{f(x^\prime+x^{\prime\prime})}
  dx^\prime
  dx^{\prime\prime}
  dk
\\
& =
  \operatorname{vol}(G_{d,n})
  \int_{\mathbb{R}^n}
  \int_{\mathcal{G}_d}
  \varphi(k\cdot\sigma,x-\pi_{k\cdot\sigma})
  \overline{f(x)}
  dk
  dx
\\
& =
  \int_{\mathbb{R}^n}
  \left(
  \operatorname{vol}(G_{d,n})
  \int_{\mathcal{G}_d}
  \varphi(k\cdot\sigma,x-\pi_{k\cdot\sigma})
  dk
  \right)
  \overline{f(x)}
  dx, 
\end{align*}
which shows \eqref{equation:adjoint1}.  
\end{proof}
The inversion formula for $\mathcal{R}_d$ is given as follows.
\begin{theorem}[{\cite[Theorem~6.2]{Siggi}}]
\label{theorem:siggi}
We have 
\begin{equation}
f=\frac{1}{\operatorname{vol}(G_{d,n})}(-\Delta_{\mathbb{R}^n})^{d/2}\mathcal{R}_d^\ast{\mathcal{R}_d}f 
\label{equation:inversion}
\end{equation}
for $f(x)=\mathcal{O}(\langle{x}\rangle^{-d-\varepsilon})$ on $\mathbb{R}^n$, 
where $\langle{x}\rangle=\sqrt{1+x_1^2+\dotsb+x_n^2}$.
\end{theorem}
We can see the results of Theorem~\ref{theorem:siggi} as that the normal operator 
$\mathcal{R}_d^\ast{\mathcal{R}_d}$ is an elliptic pseudodifferential operator $(-\Delta_{\mathbb{R}^n})^{-d/2}$, that is,
$$
\mathcal{R}_d^\ast{\mathcal{R}_d}f(x)
=
\frac{1}{(2\pi)^n}
\int_{\mathbb{R}^n}
e^{ix\cdot\xi}
\frac{\operatorname{vol}(G_{d,n})}{\lvert\xi\rvert^d}
\hat{f}(\xi)
d\xi, 
$$
where $\hat{f}(\xi)$ is the Fourier transform of $f$ defined by 
$$
\hat{f}(\xi)
=
\int_{\mathbb{R}^n}
e^{-iy\cdot\xi}
f(y)
dy.
$$
According to the theory of the composition of Fourier integral operators \cite[Theorem~25.2.3]{Hoermander4}, the symbol $\operatorname{vol}(G_{d,n})/\lvert\xi\rvert^d$ is supposed to be given by the integration of the product of the amplitudes of $\mathcal{R}_d^\ast$ and $\mathcal{R}_d$ over $e_d$-dimensional spaces. The purpose of this short paper is to observe this fact directly computing $\mathcal{R}_d^\ast\mathcal{R}_d$. 
\par
Recall that $\dim(G_{d,n})=d(n-d)$, $\dim\bigl(G(d,n)\bigr)=(d+1)(n-d)$, and 
$$
T^\ast{G(d,n)}
=
\bigl\{
\bigl(\sigma,x^{\prime\prime},(\eta_1,\dotsc,\eta_d),\xi\bigr) 
: 
\sigma \in G_{d,n},\ 
x^{\prime\prime},\eta_1,\dotsc,\eta_d,\xi\in\sigma^\perp
\bigr\}. 
$$
See \cite[Lemma~3.1]{Chihara1} for $T^\ast{G(d,n)}$. 
We denote $\sigma \in G_{d,n}$ by $\langle\omega_1,\dotsc,\omega_d\rangle_\text{ON}$ if $\sigma$ is spanned by an orthonormal system $\{\omega_1,\dotsc,\omega_d\}$ in $\mathbb{R}^n$. Strictly speaking, Fourier integral operators should be treated as linear operators acting on half-densities, but we deal with them as linear operators acting on functions in this paper for the sake of simplicity. Any problem will not occur. Basic properties of the $d$-plane transform $\mathcal{R}_d$ are the following.
\begin{theorem}[{\cite[Theorem~3.2]{Chihara1}}]
\label{theorem:chihara1} 
The $d$-plane transform $\mathcal{R}_d$ is an elliptic Fourier integral operator whose Schwartz kernel belongs to 
$$
\mathcal{I}^{-d(n-d+1)/4}\bigl(G(d,n)\times\mathbb{R}^n,\Lambda_d^\prime\bigr),
$$
where $\mathcal{I}^{-d(n-d+1)/4}(\dotsb)$ is the standard notation of the class of Lagrangian distributions {\rm (}See {\rm \cite{Hoermander4})}, $\Lambda_d$ is the canonical relation of $\mathcal{R}_d$ given by 
\begin{align*}
  \Lambda_d
& =
  \bigl\{
  \bigl((\sigma,y-\pi_\sigma{y}),\eta(y\cdot\omega_1,\dotsc,y\cdot\omega_d,1);y,\eta\bigr) \in T^\ast{G(d,n)}{\times}T^\ast\mathbb{R}^n : 
\\
& \qquad
  \sigma=\langle\omega_1,\dotsc,\omega_d\rangle_\text{ON} \in G_{d,n}, \ 
  y \in \mathbb{R}^n, \ 
  \eta \in \sigma^\perp\setminus\{0\}
  \bigr\} 
\\
& =
  \bigl\{
  \bigl((\sigma,y-\pi_\sigma{y}),\eta(y\cdot\omega_1,\dotsc,y\cdot\omega_d,1);y,\eta\bigr) \in T^\ast{G(d,n)}{\times}T^\ast\mathbb{R}^n : 
\\
& \qquad
  (y,\eta) \in T^\ast\mathbb{R}^n\setminus0, \ 
  \sigma=\langle\omega_1,\dotsc,\omega_d\rangle_\text{ON} \in G_{d,n}\cap\eta^\perp
  \bigr\} 
\\
& =
  \bigl\{
  \bigl(
  (\sigma,x^{\prime\prime},\xi(t_1,\dotsc,t_d,1);
   x^{\prime\prime}+t_1\omega_1+\dotsb+t_d\omega_d,\xi
  \bigr) 
  \in T^\ast{G(d,n)}{\times}T^\ast\mathbb{R}^n : 
\\
& \qquad
  (\sigma,x^{\prime\prime}) \in G(d,n), \ 
  \sigma=\langle\omega_1,\dotsc,\omega_d\rangle_\text{ON} \in G_{d,n}, \ 
  t_1,\dotsc,t_d \in \mathbb{R}, \ 
  \xi \in \sigma^\perp\setminus\{0\}
  \bigr\},  
\end{align*}
and $\Lambda_d^\prime$ is the twisted Lagrangian of $\Lambda_d$ defined by 
$$
\Lambda_d^\prime
:=
\bigl\{
\bigl((\sigma,x^{\prime\prime}),\Xi;y,-\eta\bigr) 
\in T^\ast{G(d,n)}{\times}T^\ast\mathbb{R}^n : 
\bigl((\sigma,x^{\prime\prime}),\Xi;y,\eta\bigr) \in \Lambda_d 
\bigr\}. 
$$
\end{theorem}
A composition $A:=A_1A_2$ of properly supported Fourier integral operators 
$A_1$ and $A_2$ is studied in H\"ormander's celebrated textbook \cite{Hoermander4}. In general, the phase function $\Phi$ of $A$ is not necessarily nondegenerate, and the degree $e$ of the degeneracy of $\Phi$ is said to be an excess. Let $a_1$, $a_2$ and $a$ be the principal symbols of $A_1$, $A_2$ and $A$ respectively, and let $\Lambda_1$, $\Lambda_2$ and $\Lambda$ be  be the canonical relations of $A_1$, $A_2$ and $A$ respectively. Then $\Lambda=\Lambda_1\circ\Lambda_2$, that is, 
$$
\Lambda
=
\Lambda_1\circ\Lambda_2
:=
\{
(x,\xi;z,\zeta) : 
\exists (y,\eta)\ \text{s.t.}\ 
(x,\xi;y,\eta)\in\Lambda_1, \ 
(y,\eta;z,\zeta)\in\Lambda_2
\}.
$$
Theorem~25.2.3 in \cite{Hoermander4} states that if $\Lambda=\Lambda_1\circ\Lambda_2$ is a clean intersection with excess $e$, then $A=A_1A_2$ is well-defined and some basic properties of the composition $A$ hold. In particular, there exist $e$-dimensional sets $\Lambda(\gamma)$, $\gamma \in \Lambda$ such that $a$ is given by 
\begin{equation}
a=\int_{\Lambda(\gamma)}a_1{\times}a_2. 
\label{equation:a1a2}
\end{equation} 
In this paper we obtain the integrals \eqref{equation:a1a2} for $\mathcal{R}_d^\ast\mathcal{R}_d$ concretely. Set 
$$
\Lambda_d^T
:=
\bigl\{
\bigl(y,\eta,(\sigma,x^{\prime\prime})\bigr) 
: 
\bigl((\sigma,x^{\prime\prime}),y,\eta\bigr)\in\Lambda_d 
\bigr\}.
$$
Let $a_d(x,\xi)$ be the symbol of $\mathcal{R}_d^\ast\mathcal{R}_d$, that is, 
$$
\mathcal{R}_d^\ast\mathcal{R}_df(x)
=
\frac{1}{(2\pi)^n}
\int_{\mathbb{R}^n}
e^{ix\cdot\xi}
a_d(x,\xi)
\hat{f}(\xi)
d\xi.
$$
Actually \eqref{equation:inversion} shows that $a_d(x,\xi)=\operatorname{vol}(G_{d,n})/\lvert\xi\rvert^d$. Let $e_d$ denote the excess of the intersection $\Lambda_d^T\circ\Lambda_d$. Actually it is well-known that 
$$
\Lambda_d^T\circ\Lambda_d
=
\Delta(T^\ast\mathbb{R}^n\setminus0)
:=
\{(y,\eta;y,\eta) : (y,\eta) \in T^\ast\mathbb{R}^n\setminus0\},
$$
which is the canonical relation of pseudodifferential operators on $\mathbb{R}^n$. 
\par
We are interested in the $e_d$-dimensional integrals \eqref{equation:a1a2} for $\mathcal{R}_d$. We now state the main theorem of this paper.  
\begin{theorem}
\label{theorem:main}
$\Lambda_d^T\circ\Lambda_d=\Delta(T^\ast\mathbb{R}^n\setminus0)$ 
is a clean intersection with excess $e_d=d(n-d-1)$. 
$\mathcal{R}_d^\ast\mathcal{R}_d$ is an elliptic pseudodifferential operator of order $-d$ on $\mathbb{R}^n$ with a symbol given 
by an $e_d$-dimensional integral
\begin{align}
  a_d(x,\xi)
& =
  \frac{\operatorname{vol}(G_{d,n})}{\lvert\xi\rvert^d}
  \int_{\bigl(O(n)/(O(d){\times}O(n-d))\bigr)\cap\xi^\perp}
  dk
\label{equation:cgamma}
\\
& =
  \frac{\operatorname{vol}(G_{d,n})}{\lvert\xi\rvert^d}
  \int_{O(n-1)/(O(d){\times}O(n-1-d))}
  dk
  =
  \frac{\operatorname{vol}(G_{d,n})}{\lvert\xi\rvert^d}.
\nonumber
\end{align}
\end{theorem}
We can also recognize that Theorem~\ref{theorem:main} gives an alternative proof of the inversion formula \eqref{equation:inversion}. 
\begin{proof}[Proof of Theorem~\ref{theorem:main}]
We show only $e_d$, the order $m$ of $\mathcal{R}_d^\ast\mathcal{R}_d$, and \eqref{equation:cgamma}. The other parts are well-known. We first compute $e_d$ and $m$. We have $\dim(\Lambda_d)=(d+1)(n-d)+n$ since $\Lambda_d$ is a conic Lagrangian submanifold of $T^\ast(G(d,n)\times\mathbb{R}^n)$, and we have $\dim(\Lambda_d^T\times\Lambda_d)=2(d+1)(n-d)+2n$. It is easy to see that 
$$
\operatorname{codim}\bigl(T^\ast\mathbb{R}^n\times\Delta(T^\ast{G(d,n)})\times{T^\ast\mathbb{R}^n}\bigr)
=
\dim(T^\ast{G(d,n)})
=
2(d+1)(n-d). 
$$
Set 
$$
E
:=
(\Lambda_d^T\times\Lambda_d)
\cap
\bigl(T^\ast\mathbb{R}^n\times\Delta(T^\ast{G(d,n)})\times{T^\ast\mathbb{R}^n}\bigr)
$$
for short. Using Theorem~\ref{theorem:chihara1}, we deduce that 
\begin{align*}
  E
& =
  \bigl\{
  \bigl(x,\xi;(\sigma,x-\pi_\sigma{x}),\xi(x\cdot\omega_1,\dotsc,x\cdot\omega_d,1);
\\
& \qquad
  (\sigma^\prime,y-\pi_{\sigma^\prime}{y}),\eta(y\cdot\omega^\prime_1,\dotsc,y\cdot\omega_d^\prime,1);y,\eta\bigr) : 
\\
& \qquad
  (x,\xi) \in T^\ast\mathbb{R}^n\setminus0, \ 
  \sigma=\langle\omega_1,\dotsc,\omega_d\rangle_\text{ON} \in G_{d,n}\cap\xi^\perp,
\\
& \qquad
  (y,\eta) \in T^\ast\mathbb{R}^n\setminus0, \ 
  \sigma^\prime=\langle\omega^\prime_1,\dotsc,\omega^\prime_d\rangle_\text{ON} \in G_{d,n}\cap\eta^\perp,
\\
& \qquad
  \sigma=\sigma^\prime,\ x-\pi_\sigma{x}=y-\pi_{\sigma^\prime}y,
\\
& \qquad 
  \xi(x\cdot\omega_1,\dotsc,x\cdot\omega_d,1)
  =
  \eta(y\cdot\omega^\prime_1,\dotsc,y\cdot\omega_d^\prime,1)
  \bigr\}.     
\end{align*}
We reduce the conditions arising in $E$ step by step. 
\begin{itemize}
\item 
Since $\sigma=\sigma^\prime$, we have $x-\pi_\sigma{x}=y-\pi_\sigma{y}$, 
and we may assume $\omega_1=\omega_1^\prime, \dotsc, \omega_d=\omega_d^\prime$.
\item 
We obtain $\xi=\eta$ from the last elements of the both hand sides of 
\begin{equation}
\xi(x\cdot\omega_1,\dotsc,x\cdot\omega_d,1)
=
\eta(y\cdot\omega^\prime_1,\dotsc,y\cdot\omega_d^\prime,1). 
\label{equation:vectors}
\end{equation}
\item 
Since $\xi=\eta\ne0$, \eqref{equation:vectors} implies that 
$x\cdot\omega_1=y\cdot\omega_1, \dotsc, x\cdot\omega_d=y\cdot\omega_d$  
and then $\pi_\sigma{x}=\pi_\sigma{y}$. Hence we deduce $x=y$ since $\pi_\sigma{x}=\pi_\sigma{y}$ and $x-\pi_\sigma{x}=y-\pi_\sigma{y}$. 
\end{itemize}
The conditions in $E$ are reduced to 
$x=y$, $\xi=\eta$ and $\sigma=\sigma^\prime$, and we obtain 
\begin{align*}
  E
& =
  \bigl\{
  \bigl(x,\xi;(\sigma,x-\pi_\sigma{x}),\xi(x\cdot\omega_1,\dotsc,x\cdot\omega_d,1);
\\
& \qquad
  (\sigma,x-\pi_\sigma{x}),\xi(x\cdot\omega_1,\dotsc,x\cdot\omega_d,1);x,\xi\bigr) : 
\\
& \qquad
  (x,\xi) \in T^\ast\mathbb{R}^n\setminus0, \ 
  \sigma=\langle\omega_1,\dotsc,\omega_d\rangle_\text{ON} \in G_{d,n}\cap\xi^\perp
  \bigr\}, 
\end{align*}
\begin{align*}
  \dim(E)
& =
  \dim(T^\ast\mathbb{R}^n\setminus0)
  +
  \dim(G_{d,n}\cap\xi^\perp)
\\
& =
  \dim(T^\ast\mathbb{R}^n)
  +
  \dim(G_{d,n-1})
  =
  2n+d(n-d-1). 
\end{align*}
We can compute 
\begin{align*}
  e_d
& =
  \operatorname{codim}(\Lambda_d^T\times\Lambda_d)
  +
  \operatorname{codim}\bigl(T^\ast\mathbb{R}^n\times\Delta(T^\ast{G(d,n)})\times{T^\ast\mathbb{R}^n}\bigr)
  -
  \operatorname{codim}(E)
\\
& =
  -\dim(\Lambda_d^T\times\Lambda_d)
  +
  \operatorname{codim}\bigl(T^\ast\mathbb{R}^n\times\Delta(T^\ast{G(d,n)})\times{T^\ast\mathbb{R}^n}\bigr)
  +
  \dim(E)
\\
& =
  -
  \{2(d+1)(n-d)+2n\}
  +
  \{2(d+1)(n-d)\}
  +
  \{2n+d(n-d-1)\}
\\
& =
  d(n-d-1),   
\end{align*}
\begin{align*}
  m
& =
  \text{order of $\mathcal{R}_d^\ast$}\ 
  +
  \ \text{order of $\mathcal{R}_d$}\ 
  +
  \frac{e_d}{2}
\\
& =
  2\times\text{order of $\mathcal{R}_d$}\ 
  +
  \frac{e_d}{2}
\\
& =
  -
  \frac{d(n-d+1)}{2}
  +
  \frac{d(n-d-1)}{2}
  =
  -d.
\end{align*}
\par
Finally we obtain \eqref{equation:cgamma}. 
Pick up arbitrary 
$\sigma=\langle\omega_1,\dotsc,\omega_d\rangle_\text{ON} \in G_{d,n}$. 
Then we have 
$k\cdot\sigma=\langle{k\omega_1,\dotsc,k\omega_d}\rangle_\text{ON}$ 
and we can introduce a coordinate system 
$$
k\cdot\sigma
=
\{t_1k\omega+t_dk\omega_d : (t_1,\dotsc,t_d)\in\mathbb{R}^d\}. 
$$
on $k\cdot\sigma$. 
We denote by $\delta_0(s)$ the Dirac measure at $s=0$. 
We deduce that 
\begin{align*}
&  \mathcal{R}_d^\ast\mathcal{R}_df(x)
\\
& =
  \operatorname{vol}(G_{d,n})
  \int_{\mathcal{G}_d}
  \mathcal{R}_df(k\cdot\sigma,x-\pi_{k\cdot\sigma}x)dk
\\
& =
  \operatorname{vol}(G_{d,n})
  \int_{\mathcal{G}_d}
  \int_{k\cdot\sigma}
  f(y^\prime+x-\pi_{k\cdot\sigma}x)
  dy^\prime
  dk
\\
& =
  \frac{\operatorname{vol}(G_{d,n})}{(2\pi)^n}
  \int_{\mathcal{G}_d}
  \int_{k\cdot\sigma}
  \int_{\mathbb{R}^n}
  e^{i(y^\prime+x-\pi_{k\cdot\sigma}x)\cdot\xi}
  \hat{f}(\xi)
  d\xi
  dy^\prime
  dk
\\
& =
  \frac{\operatorname{vol}(G_{d,n})}{(2\pi)^n}
  \int_{\mathbb{R}^n}
  \left(
  \int_{\mathcal{G}_d}
  \int_{k\cdot\sigma}
  e^{i(y^\prime+x-\pi_{k\cdot\sigma}x)\cdot\xi}
  dy^\prime
  dk
  \right)
  \hat{f}(\xi)
  d\xi
\\
& =
  \frac{\operatorname{vol}(G_{d,n})}{(2\pi)^n}
  \int_{\mathbb{R}^n}
  \left(
  \int_{\mathcal{G}_d}
  \int_{k\cdot\sigma}
  e^{i(y^\prime+x)\cdot\xi}
  dy^\prime
  dk
  \right)
  \hat{f}(\xi)
  d\xi
\\
& =
  \frac{\operatorname{vol}(G_{d,n})}{(2\pi)^n}
  \int_{\mathbb{R}^n}
  \left(
  \int_{\mathcal{G}_d}
  \int_\mathbb{R}
  \dotsb
  \int_\mathbb{R}
  e^{i(t_1k\omega_1+\dotsb+t_dk\omega_d+x)\cdot\xi}
  dt_1{\dotsb}dt_d
  dk
  \right)
  \hat{f}(\xi)
  d\xi
\\
& =
  \frac{\operatorname{vol}(G_{d,n})}{(2\pi)^n}
  \int_{\mathbb{R}^n}
  e^{ix\cdot\xi}
  \left\{
  \int_{\mathcal{G}_d}
  \left(
  \prod_{j=l}^d
  \int_\mathbb{R}
  e^{it_j(k\omega_j)\cdot\xi}
  dt_j
  \right)
  dk
  \right\}
  \hat{f}(\xi)
  d\xi
\\
& =
  \frac{\operatorname{vol}(G_{d,n})}{(2\pi)^n}
  \int_{\mathbb{R}^n}
  e^{ix\cdot\xi}
  \frac{1}{\lvert\xi\rvert^d}
  \left\{
  \int_{\mathcal{G}_d}
  \left(
  \prod_{j=l}^d
  \int_\mathbb{R}
  e^{it_j(k\omega_j)\cdot\xi/\lvert\xi\rvert}
  dt_j
  \right)
  dk
  \right\}
  \hat{f}(\xi)
  d\xi
\\
& =
  \frac{\operatorname{vol}(G_{d,n})}{(2\pi)^n}
  \int_{\mathbb{R}^n}
  e^{ix\cdot\xi}
  \frac{1}{\lvert\xi\rvert^d}
  \left(
  \int_{\mathcal{G}_d}
  \bigotimes_{j=l}^d
  \delta_0\bigl((k\omega_j)\cdot\xi/\lvert\xi\rvert\bigr)
  dk
  \right)
  \hat{f}(\xi)
  d\xi
\\
& =
  \frac{\operatorname{vol}(G_{d,n})}{(2\pi)^n}
  \int_{\mathbb{R}^n}
  e^{ix\cdot\xi}
  \frac{1}{\lvert\xi\rvert^d}
  \left(
  \int_{\mathcal{G}_d\cap\xi^\perp}
  dk
  \right)
  \hat{f}(\xi)
  d\xi,
\end{align*}
which proves \eqref{equation:cgamma}. 
This completes the proof. 
\end{proof}
%
%
%
%

\end{document}